\documentclass[11pt]{article}
\usepackage{amssymb}
\usepackage{amsthm}
\usepackage{amsmath}
\usepackage{fullpage,hyperref,algorithmic}
\usepackage{graphicx,url,amsmath,amsthm,amsfonts,amssymb,subfigure,bbm}
\usepackage{pdfsync}

\newtheorem{theorem}{Theorem}[section]
\newtheorem{lemma}[theorem]{Lemma}
\newtheorem{prop}[theorem]{Proposition}

\newtheorem{cor}[theorem]{Corollary}

\newtheorem{definition}[theorem]{Definition}

\newtheorem{remark}{Remark}[section]

\newtheorem{conjecture}[theorem]{Conjecture}
\newtheorem{question}[theorem]{Question}

\newcommand{\pr}{\mathbb{P}}

\newcommand{\jnote}[1]{}

\newcommand{\E}{{\mathbb E}}

\newcommand{\Lip}{\mathrm{Lip}}

\newcommand{\remove}[1]{}
\newcommand{\1}{\mathbf{1}}

\newcommand{\e}{\varepsilon}

\newcommand{\f}{{\varphi}}

\newcommand*{\defeq}{\mathrel{\vcenter{\baselineskip0.5ex \lineskiplimit0pt
                     \hbox{\scriptsize.}\hbox{\scriptsize.}}}%
                     =}
\newcommand{\dimAN}{\mathrm{dim}_{AN}}

\def\P{\mathbb{P}}
\begin{document}

\title{{\bf Markov type  and threshold embeddings}}

\author{Jian Ding\thanks{Research partially supported by NSF grant DMS-1313596.} \\ University of Chicago \and James R. Lee\thanks{Research partially supported
by NSF grant CCF-0915251 and a Sloan Research Fellowship.} \\ University of Washington \and Yuval Peres \\ Microsoft Research}

\date{}

\maketitle

\begin{abstract}
For two metric spaces $X$ and $Y$, say that $X$ {\em threshold-embeds into $Y$}
if there exist a number $K > 0$ and a family of
Lipschitz maps $\{\varphi_{\tau} : X \to Y : \tau > 0 \}$
such that for every $x,y \in X$,
$$
d_X(x,y) \geq \tau \implies d_Y(\varphi_{\tau}(x),\varphi_{\tau}(y)) \geq \|\varphi_{\tau}\|_{\Lip} \tau/K\,,
$$
where $\|\varphi_{\tau}\|_{\Lip}$  denotes the Lipschitz constant of $\varphi_{\tau}$.
We show that if a metric space $X$ threshold-embeds into a Hilbert space, then $X$ has Markov type 2.
As a consequence, planar graph metrics
and doubling metrics have Markov type 2, answering questions of Naor, Peres, Schramm, and Sheffield.
More generally, if a metric space $X$ threshold-embeds into a $p$-uniformly smooth Banach space,
then $X$ has Markov type $p$.
Our results suggest some non-linear analogs of Kwapien's theorem.
For instance,
a subset $X \subseteq L_1$ threshold-embeds into Hilbert space if and only if $X$ has Markov type 2.
\end{abstract}

\section{Introduction}

We begin by recalling K. Ball's notion of Markov type \cite{Ball92}.

\begin{definition}
A metric space $(X,d)$ is said to have {\em Markov type $p \in [1,\infty)$} if there is
a constant $M > 0$ such that for every $n \in \mathbb N$, the following holds.
For every reversible
Markov chain $\{Z_t\}_{t=0}^{\infty}$ on $\{1,\ldots,n\}$, every mapping $f : \{1,\ldots,n\} \to X$,
and every time $t \in \mathbb N$,
$$
\E \,d(f(Z_t), f(Z_0))^p \leq M^p t \, \E \, d(f(Z_0), f(Z_1))^p\,,
$$
where $Z_0$ is distributed according to the stationary measure of the chain.
One denotes by $M_p(X)$ the infimal constant $M$ such that the inequality holds.
\end{definition}

This is intended as a metrical generalization of the concept of {\em linear (Rademacher) type,} which we discuss shortly. One of Ball's primary motivations was
in developing non-linear analog of Maurey's extension theorem for linear operators \cite{Maurey74}.
Toward this end, he proved the following.

\begin{theorem}[\cite{Ball92}]
\label{thm:ballext}
Let $(X,d)$ be a metric space and $Y$ a Banach space.  If $X$ has Markov type 2 and $Y$ is
$2$-uniformly convex, then there exists a constant $C=C(X,Y)$ such that
for every subset $S \subseteq X$ and Lipschitz mapping $f : S \to Y$, there exists an extension $\tilde f : X \to Y$
satisfying $\tilde f|_S = f$ and $\|\tilde f\|_{\Lip} \leq C \|f\|_{\Lip}$.
\end{theorem}

Here, we use $\|f\|_{\Lip}$ to denote the infimal constant $L$ such that $f$ is $L$-Lipschitz.  The
preceding theorem is already very interesting in the case where $Y$ is a Hilbert space.
The notion of Markov type has since found a number of additional applications \cite{LMN02,BLMN05,MN06,MN12}.

Despite its apparent utility, only Hilbert spaces (and spaces which admit a bi-Lipschitz
embedding into Hilbert space) were known to have Markov type 2
until the work of \cite{NPSS06}.  The authors prove that
every $p$-uniformly smooth Banach has Markov type $p$.  Most significantly,
this implies that $L_p$ for $p > 2$ has Markov type 2, answering a fundamental open question.
Combined with Ball's work, this gives a non-linear analog
of Maurey's extension theorem in terms of uniform smoothness and
convexity of the underlying Banach spaces.
We refer to their work \cite{NPSS06} for an extended discussion.

They also establish that trees
and certain classes of Gromov hyperbolic spaces have Markov type 2.
The authors state their belief that planar graph metrics and doubling metrics
should have Markov type 2, but they are only able to show that such spaces
have Markov type $2-\varepsilon$ for every $\varepsilon > 0$.
Building on the method of \cite{NPSS06}, it can be shown that
all series-parallel graph metrics have Markov type 2 \cite{BKL07}.
Both the planar and doubling questions have recently been reiterated in the survey of Naor \cite{NaorRibe},
where the author remarks that even the special case of the
three-dimensional Heisenberg group $\mathbb H^3$ is open.
We resolve these questions and present a number of generalizations.
Our main tool in controlling Markov type is the use of embeddings that are weaker than bi-Lipschitz.

\medskip
\noindent
{\bf Threshold embeddings.}
We recall that for two metric spaces $(X,d_X)$ and $(Y,d_Y)$, a mapping
$f : X \to Y$ is called {\em bi-Lipschitz} if $f$ is invertible and
both $\|f\|_{\Lip}$ and $\|f^{-1}\|_{\Lip}$ are bounded.
The {\em distortion of $f$} is the quantity $\|f\|_{\Lip} \cdot \|f^{-1}\|_{\Lip}$.
It is straightforward that Markov type is a bi-Lipschitz invariant.
In fact, if there exists a map from $X$ into $Y$ with distortion $D$,
then manifestly, $M_p(X) \leq D \cdot M_p(Y)$.

Unfortunately, there are planar graph metrics \cite{Bourgain86,NR03,Laakso02} and doubling metrics \cite{Pansu89,Semmes96} (see
also \cite{LN06}) which do not admit any bi-Lipschitz embedding into a Hilbert space.  However,
such spaces are known to admit a weaker sort of embedding which we now recall.
Say that $X$ {\em threshold-embeds} into $Y$ if there exists a constant $K > 0$ and a family of
mappings $\{\f_{\tau} : X \to Y\}_{\tau > 0}$ such that the following holds:
For every $\tau > 0$, for every $x,y \in X$,
$$
d_X(x,y) \geq \tau \implies d_Y(\f_{\tau}(x),\f_{\tau}(y)) \geq \frac{\|\f_{\tau}\|_{\Lip}}{K} \tau\,.
$$
If we wish to emphasize the constant $K$, we will say that $X$ {\em $K$-threshold-embeds} into $Y$.
As opposed to bi-Lipschitz maps, threshold embeddings are only required to control one scale at a time.
We prove the following theorem.

\begin{theorem}\label{thm:hilbert}
If $X$ threshold-embeds into a Hilbert space, then $X$ has Markov type 2.  Quantitatively,
if $X$ admits a $K$-threshold-embedding, then $M_2(X) \leq O(K)$.
\end{theorem}

Using the known constructions of threshold embeddings for various spaces (see Section \ref{sec:embeddings}),
we confirm that a number of spaces have Markov type 2.

\begin{theorem}\label{thm:planar}
If $(X,d)$ is the shortest-path metric on a weighted planar graph, then $(X,d)$ has Markov type 2.
More generally, this holds for the shortest-path metric on any surface of bounded genus.
Quantitatively, if $X$ is the shortest-path metric on a graph of orientable genus $g > 1$, then
$M_2(X) \leq O(\log g)$.
\end{theorem}
In Section \ref{sec:embeddings}, we show that this theorem generalizes even further, to any
non-trivial minor-closed family of graphs.

\medskip

We recall that a metric space $(X,d)$ is said to be
{\em doubling with constant $\lambda$} if every bounded set in $X$ can be covered by $\lambda$
sets of half the diameter.  A space that is doubling with some constant $\lambda < \infty$
is said to be {\em doubling.}

\begin{theorem}\label{thm:doubling}
Every doubling metric space has Markov type 2.  Quantitatively, if $X$ is $\lambda$-doubling, then
$M_2(X) \leq O(\log \lambda)$.
\end{theorem}

Theorem \ref{thm:doubling} can be generalized; to this end, we now define the {\em Assouad-Nagata dimension}
of a metric space $(X,d)$.  This quantity, denoted $\dimAN(X,d)$, is the least integer $n$ such that
the following holds:  There exists a constant $c > 0$ so that for every number $r > 0$,
there is a cover $X \subseteq \bigcup_{i=1}^{\infty} U_i$ of $X$ such that each set $U_i$
has $\mathrm{diam}(U_i) \leq cr$ and every ball of radius $r$ in $X$ has non-trivial intersection
with at most $n+1$ elements of $\{U_i\}_{i=1}^{\infty}$.  Spaces of bounded Assouad-Nagata dimension
include trees, Gromov hyperbolic groups, manifolds of pinched negative sectional curvature,
Euclidean buildings, and homogeneous Hadamard manifolds (see \cite{LT05}).
In Section \ref{sec:embeddings}, we prove the following.

\begin{theorem}\label{thm:AN}
If $(X,d)$ is a metric space with finite Assouad-Nagata dimension, then $X$ has Markov type 2.
\end{theorem}

\medskip
\noindent
{\bf Uniformly smooth Banach spaces.}
The modulus of uniform smoothness of a Banach space $X$ is defined, for $\e > 0$, as
$$
\rho_X(\e) = \sup \left\{ \frac{\|x+\e y\| + \|x-\e y\|}{2} - 1 : x,y \in X, \|x\|=\|y\|=1 \right\}.
$$
The space $X$ is called {\em uniformly smooth} if $\lim_{\e \to 0} \frac{\rho_X(\e)}{\e} = 0$.
Furthermore, $X$ is said to be {\em $p$-uniformly smooth} if there is a constant $S > 0$ such that
$\rho_X(\e) \leq S^p \e^{p}$ for all $\e > 0$.
We use $S_p(X)$ to denote the infimal constant $S$ for which this holds.
It can be verified that a Banach space can only be $p$-uniformly smooth for $p \leq 2$.
Furthermore, $L_p$ is $p$-uniformly smooth for $1 \leq p \leq 2$ and $2$-uniformly smooth
for $p \geq 2$ \cite{Hanner56}; see also \cite[App. A]{BL00}.
We extend Theorem \ref{thm:hilbert} to spaces which threshold-embed into $p$-uniformly smooth
Banach spaces.

\begin{theorem}\label{thm:psmooth}
If a metric space $(X,d)$ threshold-embeds into a $p$-uniformly smooth Banach space then
$X$ has Markov type $p$.  Quantitatively if $X$ admits a $K$-threshold-embedding into a $p$-uniformly
smooth space $Y$, then $M_p(X) \leq O(K S_p(Y))$.
\end{theorem}

\medskip
\noindent
{\bf Linear type, cotype, and Kwapien's theorem.}
We now review some fundamental definitions from the geometry of Banach spaces.
A Banach space $X$ is said to have (Rademacher) type $p > 0$ if there exists a constant $T > 0$ such that
for all finite sequences $x_1, \ldots, x_n \in X$,
$$
\E \left\|\sum_{i=1}^n \e_i x_i\right\|^p \leq T^p \sum_{i=1}^n \|x_i\|^p\,,
$$
where $\{\e_1, \ldots, \e_n\}$ is an i.i.d. sequence of random signs.  The least
such constant $T$ is referred to as the type $p$ constant of $X$ and denoted $T_p(X)$.
Similarly, $X$ is said to have (Rademacher) cotype $q < \infty$ if there exists a constant $C > 0$
such that for all finite sequences $x_1, \ldots, x_n \in X$,
$$
\E \left\|\sum_{i=1}^n \e_i x_i\right\|^q \geq \frac{1}{C^q} \sum_{i=1}^n \|x_i\|^q\,.
$$
The least such constant $C$ is denoted $C_q(X)$ and called the cotype $q$ constant of $X$.

It is straightforward that any Hilbert space $H$ has type 2 and cotype 2 and, in fact, $T_2(H)=C_2(H)=1$.
A fundamental theorem of Kwapien states that, up to isomorphism, Hilbert space
is the only Banach space with these properties.

\begin{theorem}[\cite{Kwapien72}]
A Banach space has type 2 and cotype 2 if and only if it is linearly isomorphic to a Hilbert space.
Furthermore, the isomorphism constant is bounded by $T_2(X) \cdot C_2(X)$.
\end{theorem}

In line with the ``Ribe program'' (see, e.g., \cite{NaorRibe}) and the local theory of Banach spaces,
one might look for non-linear analogs of Kwapien's result.
This would involve non-linear notions of type (e.g., \cite{Ball92}) and cotype (e.g., \cite{MN08})
and a replacement of ``linear isomorphism'' by a suitable non-linear generalization.
We refer to \cite{BL00,MN08,NaorRibe} for a thorough discussion of related issues.
Unfortunately, it is folklore that the most natural generalization, in which one replaces a linear isomorphism
by a bi-Lipschitz mapping is patently false.

To see this, note that there is a family of graph metrics $\{G_k\}_{k=0}^{\infty}$ called the {\em Laakso graphs} (after \cite{Laakso02})
which bi-Lipschitz embed into $L_1$ with distortion at most 2 \cite{GNRS04}.  Furthermore, it is proved that these graphs have Markov type 2 with uniform constant \cite{NPSS06},
i.e. $\sup_{k \geq 1} M_2(G_k) < \infty$.  Since $L_1$ has cotype 2 (see, e.g., \cite{LT79}), a straightforward non-linear Kwapien theorem
would state that they admit bi-Lipschitz embeddings into a Hilbert space with uniformly bounded distortion, but this is known to be impossible \cite{Laakso02}.
Note that one can easily construct a single infinite subset of $L_1$ (by taking a suitable infinite union of the graphs) which
has Markov type 2 but admits no bi-Lipschitz embedding into a Hilbert space.

We propose that the correct analog of ``linear isomorphism'' in the setting of Kwapien's theorem is the notion of
a threshold embedding.  To this end,
one should first observe the following.

\begin{theorem}\label{thm:equiv}
A Banach space $X$ threshold-embeds into a Hilbert space if and only if it is linearly isomorphic to a Hilbert space.
\end{theorem}

Initially, we proved this using Theorem \ref{thm:hilbert}, but Assaf Naor pointed out to us
that it follows in a simpler way using the notion of Enflo type, which was introduced in \cite{Enflo70}.
A metric space $(X,d)$ is said to have {\em Enflo type $p$} if there is a constant $E > 0$ such that
for every $n \in \mathbb N$ and every mapping $f : \{0,1\}^n \to X$, we have the inequality
$$
\sum_{\stackrel{x,y \in \{0,1\}^n}{\|x-y\|_1 = n}} d(f(x),f(y))^p \leq E^p \sum_{\stackrel{x,y \in \{0,1\}^n}{\|x-y\|_1 = 1}} d(f(x),f(y))^p\,.
$$
In this case, one says that $(X,d)$ has Enflo type $p$ with constant $E$.
It is known that for any metric space, Markov type $p$ implies Enflo type $p$ \cite[Prop. 1]{NS02}.

\begin{prop}[Naor, personal communication]
\label{prop:enflo}
Suppose $(X,d_X)$ and $(Y,d_Y)$ are metric spaces and $Y$ has Enflo type $p$ with constant $E$.
If $X$ $K$-threshold-embeds into $Y$, then $X$ has Enflo type $p$ with constant $O(KE)$.
\end{prop}

Using the preceding result along with a theorem of \cite{AMM85} and Kwapien's theorem,
Theorem \ref{thm:equiv} follows readily.  We refer to Section \ref{sec:kwapien}.

\medskip


While we are not able to give a full non-linear analog of Kwapien's theorem, we do take some steps
in this direction.  In particular, our methods are strong enough to give such a result for subsets of $L_1$.
Recall that a {\em uniform embedding} $f : X \to Y$ between metric spaces $X$ and $Y$ is an invertible
map such that $f$ and $f^{-1}$ are both uniformly continuous.  It is known that if a Banach space $X$
admits a uniform embedding into a Hilbert space, then $X$ has cotype 2, but
the converse is not true, even for spaces of non-trivial type (see \cite[\S 8.2]{BL00}).
We prove the following first step.

\begin{theorem}
Suppose that $X$ is a Banach space that admits a uniform embedding into a Hilbert space.
Then a subset $S \subseteq X$ threshold-embeds into Hilbert space if and only if $S$ has Markov type 2.
\end{theorem}
In particular, it is well-known that $L_1$ uniformly embeds into $L_2$, thus a subset of $L_1$ threshold-embeds
into Hilbert space if and only if it has Markov type 2.
There seem to be two non-trivial steps in completing a non-linear Kwapien theorem.  The
first involves metric subsets of linear spaces.

\begin{conjecture}
If $X$ is a Banach space of cotype 2, then a subset $S \subseteq X$ threshold-embeds
into Hilbert space if and only if $S$ has Markov type 2.
\end{conjecture}

We are not able to  resolve the validity of the
conjecture for the spaces of Schatten class operators $C_p$ for
$1 \leq p < 2$ (see \cite[\S 8.2]{BL00}).
A truly satisfactory non-linear Kwapien theorem would involve no reference to linear spaces at all; it would instead rely on an appropriate
notion of metric cotype.

\begin{question}
Suppose that $(X,d)$ is a metric space of metric cotype 2 (in the sense of \cite{MN06}) and Markov type 2.
Does this imply that $X$ threshold-embeds into Hilbert space?
\end{question}

Part of the other side of this question is open as well.
By the non-linear Maurey-Pisier Theorem \cite{MN06}, we know that if $X$ threshold-embeds
into a Hilbert space then it has finite metric cotype, but more should be true.

\begin{question}
If $(X,d)$ threshold-embeds into Hilbert space, does this imply that $X$ has metric cotype 2?
\end{question}

This would imply, in particular, that planar and doubling metrics have metric cotype 2,
answering a question of \cite[\S 4]{NaorRibe}.

\medskip
\noindent
{\bf Finite metric spaces and some historical remarks.}
Threshold embeddings have been studied in the context of embeddings
of finite metric spaces into Banach spaces.  Rao \cite{Rao99} showed
that finite planar graph metrics admit $O(1)$-threshold-embeddings
into Euclidean space in his proof of a multi-commodity max-flow/min-cut theorem.
More generally, using the results of \cite{KPR93}, he proved
this for any family of graphs excluding a fixed minor.
On the other hand, Bourgain \cite{Bourgain86}  exhibited an infinite family of (finite) trees
that admit no bi-Lipschitz embedding with uniformly bounded distortion into a Hilbert space.

It was earlier shown by Semmes \cite{Semmes96}, using an important result of Pansu \cite{Pansu89}, that the 3-dimensional Heisenberg group $\mathbb H^3$
(which is doubling) does not admit a bi-Lipschitz embedding into any finite-dimensional Euclidean space.  Since Pansu's technique
can be extended to any Banach space having the Radon-Nikodym property (see \cite{LN06}), this yields a metric
space which threshold-embeds into Hilbert space, but does not bi-Lipschitz embed into any Banach space with the RNP.
An example of \cite{Laakso02} gives an infinite doubling, planar graph metric which does not
bi-Lipschitz embed into any uniformly convex Banach space.  It was an open problem (of  relevance to
applications in theoretical computer science) to determine whether any metric space which threshold-embeds
into Hilbert space admits a bi-Lipschitz embedding into $L_1$.  This was resolved negatively
by Cheeger and Kleiner \cite{CK10} who showed that $\mathbb H^3$
does not bi-Lipschitz embed into $L_1$.

Returning to finite metric spaces, the utility of a family of mappings for each scale was made explicit in \cite{KLMN05}
where the authors use this approach to give a new proof of Bourgain's theorem \cite{Bourgain85}
on embedding of finite metric spaces into Hilbert space.  In \cite{Lee05},
it is proved that if an $n$-point metric space $X$ $K$-threshold-embeds
into $L_p$ for $p \geq 2$, then $X$ bi-Lipschitz embeds into $L_p$
with distortion $O(K^{1-1/p} (\log n)^{1/p})$.  In particular, such a space
admits a distortion $O(K)$ embedding into $L_{p}$ for some $p = O(\log \log n)$.

In \cite{NPSS06}, what we call threshold embeddings are referred to as ``weak embeddings.''
The authors also define a notion of ``weak Markov type 2'' and show that this property readily
follows from the existence of a threshold embedding into Hilbert space.

\subsection{Outline of our approach}

Let $\mathcal Z$ be a normed space and
consider a Markov chain $\{Z_t\}_{t=0}^{\infty}$ on a finite state space $\Omega$
and a mapping $f : \Omega \to \mathcal Z$.  In \cite{NPSS06}, it is shown that
for every time $t \in \mathbb N$, there exist martingales $\{M_k\}_{k=0}^{t}$ and $\{N_k\}_{k=0}^t$ such that
\begin{equation}\label{eq:NPSS}
f(Z_{2t})-f(Z_{0}) = M_t - N_t\,,
\end{equation}
where $\{M_k\}$ and $\{N_k\}$ both naturally trace the
evolution of $\{f(Z_k)\}$ forward in time and backward in time, respectively.
The decomposition is reviewed in Section \ref{sec:decomp}.
This allows one to reduce various problems on Markov chains to
potentially easier problems on martingales.

Now consider a metric space $(X,d)$ and
a threshold embedding $\{ \f_{\tau} : X \to \mathcal Z : \tau > 0\}$ of $X$ into $\mathcal Z$.
In determining the Markov type of $(X,d)$, one is naturally led to study mappings $g : \Omega \to X$
via the composition maps $\f_{\tau} \circ g : \Omega \to \mathcal Z$.  But crucially,
the martingales $\{M_k\}$ and $\{N_k\}$ from \eqref{eq:NPSS} depend heavily on the map $f$,
and thus the problem of deducing Markov type from a threshold embedding
becomes one of controlling an entire family of martingales---one for every map $\f_{\tau} \circ g$ for $\tau > 0$.

Fortunately, when one allows the map $f$ in \eqref{eq:NPSS} to vary, all the martingales that arise
are defined with respect to the same pair of filtrations, and their differences can be uniformly controlled
in terms of jumps of the chain $\{f(Z_k)\}$.
This leads to a problem on simultaneously bounding the tail of all martingales
whose differences are subordinate to a common sequence of random variables.
We present a representative lemma of this form, for the case of real-valued martingales.

\begin{lemma}\label{lem:rval}
There exists a constant $K > 0$ such that the following holds.
Let $\{\mathcal F_t\}$ be a filtration, and let $\{\alpha_t\}$ be a sequence adapted to $\{\mathcal F_t\}$.
Let $$\left\{\vphantom{\bigoplus}\{M_t^{\xi}\} : \xi \in I\right\}$$ be a countable collection of real-valued
martingales with respect to $\{\mathcal F_t\}$ such that
$|M^{\xi}_t - M^{\xi}_{t-1}| \leq \alpha_t$ for all $t \geq 1$.  Then for every $n \geq 0$, we have
$$\int_0^\infty y \sup_{\xi\in I} \P(|M_n^\xi - M_0^{\xi}| \geq y)\, dy \leq K \sum_{t=1}^n \E (\alpha_t^2)\,.$$
\end{lemma}

A generalization of this result to martingales taking values in uniformly smooth Banach spaces
appears as Lemma \ref{cor-integrate-tail}.  Our approach to Lemma \ref{lem:rval} is via
classical distributional inequalities
that allow one to control the maximum process associated to certain martingales
in terms of the corresponding square functions; see Burkholder's survey \cite{Burk73}.

The extension of Lemma \ref{lem:rval} to $p$-uniformly smooth Banach spaces can be deduced from
the real case using a ``dimension reduction''
lemma
for martingales in such spaces; this is inspired by
a Hilbert-space version due to Kallenberg and Sztencel \cite{KS91} (see Lemma \ref{lem:KS}).
Our dimension reduction arguments
appear in Section \ref{sec:unismooth} and may be of independent interest.
In Section \ref{sec:burk}, we present a more direct proof of Lemma \ref{cor-integrate-tail}
due to Adam Osekowski.
That proof avoids dimension reduction
and instead employs Pisier's martingale inequality for $p$-uniformly smooth Banach spaces \cite{Pisier75}.

With these tools in hand, the proof of Theorem \ref{thm:psmooth} is carried out in Section \ref{sec:threshold}.
In Section \ref{sec:embeddings}, we recall how one constructs threshold-embeddings into Hilbert space
using random partitions, yielding the proofs of Theorems \ref{thm:planar}, \ref{thm:doubling}, and \ref{thm:AN}.
Finally, in Section \ref{sec:kwapien} we discuss
the possibility of a non-linear version of Kwapien's theorem.

\section{Distributional inequalities for martingales}
\label{sec:burk}

We first recall a martingale inequality of Pisier \cite{Pisier75}.

\begin{lemma}\label{lem:pisier}
There is exists a constant $L > 0$ such that the following holds.
If $1 < p \leq 2$ and $\mathcal Z$ is a $p$-uniformly smooth Banach space,
then for any $\mathcal Z$-valued martingale $\{M_t\}_{t=0}^n$,
$$
\E\,\|M_n-M_0\|^p \leq L S_p(\mathcal Z) \sum_{t=0}^{n-1} \E\,\|M_{t+1}-M_t\|^p\,.
$$
\end{lemma}

In this section, we will write $a \vee b$ for $\max(a,b)$.
The following is a key estimate.

\begin{lemma}\label{lem-Osekowski}
For $1<p\leq 2$, let $\mathcal Z$ be a $p$-uniformly smooth Banach space, and let $\{M_t\}_{t=0}^n$ be a martingale on $\mathcal{Z}$ with respect to the filtration $\{\mathcal F_t\}$.
Denote $M^* = \max_{1\leq t\leq n}\|M_t -M_0\|$,  $\Delta^* = \max_{1\leq t\leq n} \|M_t- M_{t-1}\|$, and $$\Gamma = \left(\sum_{t=1}^n \E \left[\|M_t - M_{t-1}\|^p \mid \mathcal F_{t-1}\right]\right)^{1/p}\,.$$
Then for each $\lambda >0$, $\beta>1$ and $\delta\in(0, \beta -1)$, we have
$$\P(M^* \geq \beta \lambda) \leq \left(\frac{K \delta}{\beta - \delta - 1}\right)^p \P(M^* \geq \lambda) + \P(\Gamma \vee \Delta^* \geq \delta \lambda)\,,$$
where  $K = O(S_p(\mathcal Z))$.
\end{lemma}

\begin{proof}
We may assume that $M_0 = 0$.
For $1\leq \ell\leq n$, write $\Gamma(\ell) = (\sum_{t=1}^\ell \E [\|M_t - M_{t-1}\|^p \mid \mathcal F_{t-1})]^{1/p}$.
With the convention that $\min \emptyset = \infty$ and $\Gamma(n+1) = \Gamma(n)$, define the stopping times
\begin{align*}
\mu &= \min\{1\leq t\leq n: \|M_t\| \geq \lambda\}\,, \\
\nu &= \min\{1\leq t\leq n: \|M_t\| \geq \beta \lambda\}\,,\\
\sigma & = \min\{1\leq t\leq n: \|M_t - M_{t-1}\| \vee \Gamma(t+1) \geq \delta \lambda\}\,.
\end{align*}
Note that $\sigma$ is a stopping time since $\Gamma(t)$ is measurable with respect to $\mathcal F_t$.
In addition, define $$H_t = \sum_{j=1}^t \mathbf{1}_{\{\mu<j\leq \nu\wedge\sigma\wedge n\}} (M_j - M_{j-1})\,,$$
where we have used $\1_E$ to denote the indicator of the event $E$.

Observe that $\{H_t\}$ is also a martingale. By Doob's maximal inequality (see, e.g., \cite[\S 4.4]{Durett96}), we have
\begin{align}
\P(M^* \geq \beta \lambda, \Gamma \vee \Delta^* \leq \delta \lambda) &= \P(\mu \leq \nu \leq n, \sigma = \infty) \nonumber\\
& \leq \P\left(\max_{1\leq t\leq n} \|H_t\| \geq (\beta - 1 - \delta)\lambda\right) \nonumber \\
& \leq
\frac{\E\, \|H_n\|^p}{((\beta - 1- \delta) \lambda)^p}\,.\label{eq-adam-1}
\end{align}
Applying Lemma \ref{lem:pisier}, one obtains
\begin{align*}\E\, \|H_n\|^p &\leq O(S_p(\mathcal{Z})^p)\, \E \left(\sum_{t=1}^n \E\left[\|H_t - H_{t-1}\|^p \mid \mathcal F_{t-1}\right] \right) \\
 &\leq O(S_p(\mathcal{Z})^p)\, \E \left[\Gamma(\sigma)^p \mathbf{1}_{\mu < \infty}\right]\\
& \leq O(S_p(\mathcal{Z})^p) \,(\delta \lambda)^p \,\P(M^* \geq \lambda)\,.
\end{align*}
Combining this with \eqref{eq-adam-1}, the desired conclusion follows.
\end{proof}

The preceding lemma leads to the following consequence.
\begin{lemma}\label{cor-integrate-tail}
For $p \in (1,2]$, let $\mathcal Z$ be a $p$-uniformly smooth Banach space.
There exists a constant $C = O(S_p(\mathcal Z)^p)$ such that the following holds.
Let $\{\mathcal F_t\}$ be a filtration, and let $\{\alpha_t\}$ be a sequence adapted to $\{\mathcal F_t\}$.
Let $$\left\{\vphantom{\bigoplus}\{M_t^{\xi}\} : \xi \in I\right\}$$ be a countable collection of $\mathcal Z$-valued
martingales with respect to $\{\mathcal F_t\}$ such that
$\|M^{\xi}_t - M^{\xi}_{t-1}\| \leq \alpha_t$ for all $t \geq 1$.  Then for every $n \geq 0$, we have
$$\int_0^\infty y^{p-1} \sup_{\xi\in I} \P(\|M_n^\xi - M_0^{\xi}\| \geq y)\, dy \leq C\sum_{t=1}^n \E (\alpha_t^p)\,.$$
\end{lemma}
\begin{proof}
For $\xi\in I$, we make the definitions
 \begin{eqnarray*}
 M^*_\xi &=& \max_{1\leq t\leq n} \|M_t^\xi- M_0^\xi\| \\
 \Delta^*_\xi &=& \max_{1\leq t\leq n} \|M_t^\xi - M_{t-1}^\xi\| \\
 \Gamma_\xi &=& \left(\sum_{t=1}^n \E\left[\|M_t^\xi - M_{t-1}^\xi\|^p \mid \mathcal F_{t-1}\right]\right)^{1/p}\,.
 \end{eqnarray*}
Applying Lemma~\ref{lem-Osekowski} to each martingale $\{M_t^\xi\}$, we obtain
$$\P(M^*_\xi \geq \beta \lambda) \leq \left(\frac{K\delta}{\beta - \delta -1}\right)^p \P(M^*_\xi \geq \lambda) + \P(\Gamma_{\xi} \vee \Delta_\xi^* \geq \delta \lambda)\,.$$
Write $\Gamma_\alpha = (\sum_{t=1}^n \E (\alpha_t^p\mid \mathcal F_{t-1}))^{1/p}$ and $\alpha^* = \max_{1\leq t\leq n} \|\alpha_t\| $. Since $\Gamma_\xi \vee \Delta^*_{\xi} \leq \Gamma_\alpha \vee \alpha^*$, we have
$$\P(M^*_\xi \geq \beta \lambda) \leq \left(\frac{K\delta}{\beta - \delta -1}\right)^p \P(M^*_\xi \geq \lambda) + \P(\Gamma_{\alpha} \vee \alpha^* \geq \delta \lambda)\,.$$
Therefore,
$$\sup_{\xi\in I}\P(M^*_\xi \geq \beta \lambda) \leq \left(\frac{K\delta}{\beta - \delta -1}\right)^p \sup_{\xi\in I}\P(M^*_\xi \geq \lambda) + \P(\Gamma_{\alpha} \vee \alpha^* \geq \delta \lambda)\,.$$
Multiplying both sides by $\beta^p \lambda^{p-1}$ and integrating in $\lambda$, we obtain
$$\left(1 - \left(\frac{\beta \delta K}{\beta -1 -\delta}\right)^p\right) \int_0^\infty \sup_{\xi \in I} \P(M_\xi^* \geq \lambda)\, d\lambda \leq \frac{\beta^p}{\delta^p} \E \left[ (\Gamma_\alpha\vee \alpha^*)^p\right] \leq  \frac{2\beta^p}{\delta^p} \E \left[\Gamma_\alpha^p\right]\,.$$
Choosing $\beta = 5$ and $\delta = (2 (K \vee 1))^{-1}$, this completes the proof of the lemma.
\end{proof}

\section{Markov type and threshold embeddings}

With Lemma \ref{cor-integrate-tail} in hand, we are ready to prove our main theorem.
We first review the decomposition of a Markov chain on a normed
space into a pair of martingales.  Then in Section \ref{sec:threshold}, we
relate Markov type and threshold embeddings into uniformly smooth spaces.
Finally, in Section \ref{sec:embeddings}, we use this to endow certainly
families of metric spaces with Markov type 2.

\subsection{The martingale decomposition}
\label{sec:decomp}

Let $\mathcal Z$ be a normed space, and let $\{Z_s\}_{s=0}^{\infty}$ be a stationary, reversible Markov chain on a finite
state space $\Omega$.   Consider any $f : \Omega \to \mathcal Z$ and $t \in \mathbb N$.
Define the martingales $\{M_s\}_{s=0}^t$ and $\{N_s\}_{s=0}^{t}$ by $M_0 = f(Z_0)$ and $N_0 = f(Z_t)$ and
for $0 \leq s \leq t-1$,
\begin{eqnarray}
M_{s+1} - M_s &\defeq & f(Z_{s+1}) -  f(Z_s) - \E \left[f(Z_{s+1})-f(Z_s) \mid Z_s\right] \label{eq:Mdef} \\
N_{s+1} - N_s &\defeq & f(Z_{t-s-1}) -  f(Z_{t-s}) - \E \left[f(Z_{t-s-1})-f(Z_{t-s}) \mid Z_{t-s}\right]\,. \label{eq:Ndef}
\end{eqnarray}

Observe that $\{M_s\}$ is a martingale with respect to the filtration induced on $\{Z_0, Z_1, \ldots, Z_t\}$
and $\{N_s\}$ is a martingale with respect to the filtration induced on $\{Z_t, Z_{t-1}, \ldots, Z_0\}$.
For every $1 \leq s \leq t-1$, one can use stationarity and reversibility to verify that
\begin{equation}\label{eq:parity}
f(Z_{s+1}) - f(Z_{s-1}) = (M_{s+1} - M_{s-1}) - (N_{t-s+1}-N_{t-s})\,,
\end{equation}
since $\E[f(Z_{s+1}) \mid Z_s] = \E[f(Z_{s-1}) \mid Z_s]$.

We also define the martingales $\{A_k\}_{0 \leq k \leq t/2}$ and $\{B_k\}_{0 \leq k \leq t/2}$ by
\begin{eqnarray*}
A_k &\defeq & \sum_{s=1}^k M_{2s} - M_{2s-1} \\
B_k &\defeq & \sum_{s=1}^k N_{2s} - N_{2s-1}\,.
\end{eqnarray*}
If we assume that $t$ is even, then summing \eqref{eq:parity} over $s=1,3,5,\ldots,t/2-1$ yields
\begin{equation}\label{eq:AB}
f(Z_t) - f(Z_{0}) = A_{t/2} - B_{t/2}\,.
\end{equation}
Finally, we observe that for any $1 \leq s \leq t/2$ and $p \geq 1$, we have the inqualities
\begin{eqnarray}
\!\!\!\!\!\!\!\|A_{s}-A_{s-1}\|^p  &\leq& 2^{p-1} \|f(Z_{2s})-f(Z_{2s-1})\|^p \label{eq:Asteps}
+ 2^{p-1} \E \left[\vphantom{\bigoplus} \|f(Z_{2s})-f(Z_{2s-1})\|^p \mid Z_{2s-1} \right] \\
\!\!\!\!\!\!\!\|B_{s}-B_{s-1}\|^p  &\leq& 2^{p-1} \|f(Z_{t-2s+1})-f(Z_{t-2s})\|^p
+ 2^{p-1} \E \left[\vphantom{\bigoplus} \|f(Z_{t-2s+1})-f(Z_{t-2s})\|^p \mid Z_{t-2s} \right] \label{eq:Bsteps}
\end{eqnarray}
These follow from $A_s - A_{s-1} = M_{2s} - M_{2s-1}$ and $B_s - B_{s-1} = N_{2s} - N_{2s-1}$
along with the definitions \eqref{eq:Mdef} and \eqref{eq:Ndef}.

\remove{
\bigskip

For each $j \in \mathbb Z$, we know there are martingales $\{M^{(j)}_t\}$ and $\{N^{(j)}_t\}$
such that
$$
f_j(Z_0) - f_j(Z_{2t}) = \sum_{k=1}^t (M^{(j)}_{2k} - M^{(j)}_{2k-1}) - \sum_{k=1}^t (N^{(j)}_{2k} - N^{(j)}_{2k-1})\,.
$$
Let $A^{(j)}_s = \sum_{k=1}^s (M^{(j)}_{2k} - M^{(j)}_{2k-1})$ and
$B^{(j)}_s = \sum_{k=1}^s (N^{(j)}_{2k} - N^{(j)}_{2k-1})$.
Note that $A^{(j)}_s$ is a martingale with respect to $\{Z_0,Z_2, \ldots, Z_{2s-1}\}$
and $B^{(j)}_s$ is a martingale with respect to $\{Z_{2t-2s+1}, Z_{2t-2s+2}, \ldots, Z_{2t}\}$.

Also, observe that
\begin{eqnarray*}
A^{(j)}_s - A^{(j)}_{s-1} &=& M^{(j)}_{2s} - M^{(j)}_{2s-1} \\
B^{(j)}_s - B^{(j)}_{s-1} &=& N^{(j)}_{2s} - N^{(j)}_{2s-1}\,. \\
\end{eqnarray*}
}

\subsection{Threshold embeddings}
\label{sec:threshold}

\begin{theorem}\label{thm:main}
For $p \in (1,2]$, let $\mathcal Z$ be a $p$-uniformly smooth Banach space.
If $(X,d)$ is a metric space
that threshold-embeds into $\mathcal Z$, then $X$ has Markov type $p$.
Quantitatively, if $X$ $D$-threshold-embeds into $\mathcal Z$, then
$M_p(X) \leq O(D S_p(\mathcal Z))$.
\end{theorem}

\begin{proof}
Let $\{\varphi_{\tau} : X \to \mathcal Z : \tau \geq 0\}$ be a
family of 1-Lipschitz mappings which satisfy $$d(x,y) \geq \tau \implies \|\varphi_{\tau}(x)-\varphi_{\tau}(y)\| \geq \tau/D$$
for all $x,y \in X$.

Consider a finite state space $\Omega$, a mapping $g : \Omega \to X$, and a stationary, reversible
Markov chain $\{Z_s\}_{s=0}^{\infty}$ on $\Omega$.
We assume that $Z_0$ is distributed according to the stationary measure.
Fix an even number $t=2u$.
For each $j \in \mathbb Z$, let $\{A^{(j)}_s\}$ and $\{B^{(j)}_s\}$ be
the martingales from Section \ref{sec:decomp} corresponding to the choice $f = \f_{2^j} \circ g : \Omega \to \mathcal Z$.

From \eqref{eq:AB}, we have $\f_{2^j}(g(Z_0)) - \f_{2^j}(g(Z_{t})) = A^{(j)}_{u} - B^{(j)}_{u}$.
Thus we can write
\begin{eqnarray*}
\E \,d(g(Z_0),g(Z_{t}))^p &=&  p \int_0^{\infty} \lambda^{p-1} \cdot \pr(d(g(Z_0), g(Z_{t})) \geq \lambda )\,d\lambda \\
&\leq & p \sum_{j\in \mathbb Z} 2^{(j+1)(p-1)+j} \cdot \pr(d(g(Z_0), g(Z_{t})) \geq 2^j ) \\
&\leq & p\sum_{j \in \mathbb Z} 2^{(j+1)(p-1)+j} \cdot \pr\left(\|\f_{2^j}(g(Z_0))-\f_{2^j}(g(Z_{t}))\| \geq \frac{2^j}{D}\right) \\
&=&
p\sum_{j \in \mathbb Z} 2^{(j+1)(p-1)+j} \cdot \pr\left(\|A^{(j)}_{u}-B^{(j)}_{u}\| \geq \frac{2^j}{D} \right) \\
&\leq &
p\sum_{j \in \mathbb Z} 2^{(j+1)(p-1)+j} \cdot \left[\pr\left(\|A^{(j)}_{u}\| > \frac{2^{j-1}}{D} \right) + \pr\left(\|B^{(j)}_{u}\| > \frac{2^{j-1}}{D} \right) \right]\\
&\leq&
p 2^{3p-1} D^p  \left(\int_{0}^{\infty} y^{p-1} \sup_{j \in \mathbb Z} \pr\left(\|A^{(j)}_u\| > y\right) \,dy
+
\int_{0}^{\infty} y^{p-1} \sup_{j \in \mathbb Z} \pr\left(\|B^{(j)}_u\| > y\right) \,dy\right)\,.
\end{eqnarray*}

Now define, for $1 \leq s \leq u$,
\begin{eqnarray*}
\alpha_s &\defeq& \left(\vphantom{\bigoplus} 2^{p-1}\,d(g(Z_{2s}), g(Z_{2s-1}))^p + 2^{p-1} \E[d(g(Z_{2s}), g(Z_{2s-1}))^p \mid Z_{2s-1}]\right)^{1/p} \\
\beta_s &\defeq&  \left(\vphantom{\bigoplus} 2^{p-1}\,d(g(Z_{t-2s+1}), g(Z_{t-2s}))^p + 2^{p-1} \E[d(g(Z_{t-2s+1}), g(Z_{t-2s}))^p \mid Z_{t-2s}]\right)^{1/p}\,,
\end{eqnarray*}
Then, using stationarity, we have the bounds
\begin{equation}\label{eq:stationary}
\E \left[\sum_{s=1}^{u} \alpha_s^p\right], \E \left[\sum_{s=1}^{u} \beta_s^p\right] \leq 4u \E\,d(g(Z_0),g(Z_1))^p = 2t \E\,d(g(Z_0),g(Z_1))^p\,.
\end{equation}
Additionally, by \eqref{eq:Asteps} and \eqref{eq:Bsteps} and the fact that $\f_{2^j}$ is 1-Lipschitz for each $j \in \mathbb Z$,
for the range $1 \leq s \leq u$, we have
\begin{eqnarray*}
\|A^{(j)}_s - A^{(j)}_{s-1}\| &\leq& \alpha_s \\
\|B^{(j)}_s - B^{(j)}_{s-1}\| &\leq& \beta_s \,.
\end{eqnarray*}

We can thus apply Lemma \ref{cor-integrate-tail} to conclude that
\begin{eqnarray*}
\int_{0}^{\infty} y^{p-1} \sup_{j \in \mathbb Z} \pr\left(\|A^{(j)}_u\| > y\right) \,dy &\leq& K \E\left[\sum_{s=1}^u \alpha_s^p\right] \\
\int_{0}^{\infty} y^{p-1} \sup_{j \in \mathbb Z} \pr\left(\|B^{(j)}_u\| > y\right) \,dy &\leq& K \E\left[\sum_{s=1}^u \beta_s^p\right],
\end{eqnarray*}
where $K = O(S_p(\mathcal Z)^{p/2})$.
Combining these estimates with \eqref{eq:stationary} and our previous discussion,
for every even time $t$, we have
$$
\E\left[\vphantom{\bigoplus}d(g(Z_0), g(Z_t))^p\right] \leq 4K p 2^{3p-1} D^p  t \,\E\left[\vphantom{\bigoplus} d(g(Z_0),g(Z_1))^p\right]\,.
$$

Finally, if $t$ is odd, then
\begin{eqnarray*}
\E\left[\vphantom{\bigoplus}d(g(Z_0), g(Z_t))^p\right] &\leq& 2^{p-1} \E\left[\vphantom{\bigoplus}d(g(Z_0), g(Z_{t-1}))^p\right] + 2^{p-1} \E\left[\vphantom{\bigoplus}d(g(Z_{t-1}), g(Z_t))^p\right] \\
&\leq &
8K  p 2^{3p-1} D^p  t \,\E\left[\vphantom{\bigoplus} d(g(Z_0),g(Z_1))^p\right]\,.
\end{eqnarray*}
Thus $(X,d)$ has Markov type $p$ with constant $M_p(X) = O(D S_p(\mathcal Z))$.
\end{proof}
\subsection{Random partitions}
\label{sec:embeddings}

We now recall that planar metrics, doubling metrics, and more general families
of metric spaces admit threshold embeddings into Hilbert space,
and therefore, by our main theorem, have Markov type 2.

Let $(X,d)$ be a metric space.
If $P$ is a partition of $X$ and $x \in X$,
we will write $P(x)$ for the unique set of $P$ containing $x$.
A {\em random partition of $X$} is a probability
space $(\Omega,\Sigma,\mu)$, together with
a mapping $\omega \mapsto P_{\omega}$ which
associates to every $\omega \in \Omega$
a partition $P_{\omega}$ of $X$.
We will use $\mathcal P$ to denote such a random partition.

To sidestep issues of measurability, we will assume that
every random partition
$\mathcal P$ is supported on only countably many partitions,
each of which is composed of only countably many sets of $X$.
This presents no difficulty because it will
hold in all the cases that arise.  We refer to \cite{LN06}
for a more detailed discussion.

We say that $\mathcal P$ is {\em $\Delta$-bounded}
if
$$
\P[\forall S \in \mathcal P, \mathrm{diam}(S) \leq \Delta] = \mu\left(\left\{\omega : \forall S \in P_{\omega}, \mathrm{diam}(S) \leq \Delta\right\}\right) = 1\,.
$$
We say that $\mathcal P$ is {\em $(\e,\delta,\Delta)$-padded} if $\mathcal P$ is $\Delta$-bounded, and for all $x \in X$,
$$
\P[B(x,\e \Delta) \subseteq \mathcal P(x)] = \mu\left(\left\{\omega : B(x,\e \Delta) \subseteq P_{\omega}(x) \right\}\right) \geq \delta\,,
$$
where $B(x,R) = \{ y \in X : d(x,y) \leq R \}$ denotes the closed ball around $x$.
The next theorem follows from the techniques of \cite{Rao99} and \cite{LMN05}.
The result is well-known, but we could not find the theorem stated
explicitly, so we prove it here.

\begin{theorem}\label{thm:partitions}
Consider a metric space $(X,d)$ and numbers $\e,\delta > 0$.  Suppose that
for every $\Delta > 0$, $X$ admits an $(\e,\delta,\Delta)$-padded random partition.
Then $(X,d)$ $K$-threshold embeds into Hilbert space, where $K \leq \frac{4}{\e\sqrt{\delta}}$.
\end{theorem}

\begin{proof}
Fix $\tau > 0$ and put $\Delta = \tau/2$.  We may assume that $\Delta < \mathrm{diam}(X)$.
 Let $\mathcal P$ be a $(\e,\delta,\Delta)$-padded random partition
with associated probability space $(\Omega,\Sigma,\mu)$.  For each set $S$ which occurs as a member of $\bigcup_{\omega \in \Omega} P_{\omega}$,
let $\sigma_S$ be an independent Bernoulli $\{0,1\}$ random variable.  We use $(\Omega', \P)$ to denote the product
space encompassing $(\Omega,\mu)$ and $\{\sigma_S\}$.  Finally, consider the map
$\f_{\tau} : X \to L_2(\Omega',\P)$ given by
$$
\f_{\tau}(x) = \sigma_{\mathcal P(x)} \cdot d(x, X \setminus \mathcal P(x))\,.
$$
Observe that for $x,y \in X$,
we have
$$
\|\f_{\tau}(x) - \f_{\tau}(y)\|^2_{L_2(\Omega',\P)} =\E\, |\sigma_{\mathcal P(x)} \cdot d(x, X \setminus \mathcal P(x)) - \sigma_{\mathcal P(y)} \cdot d(y, X \setminus \mathcal P(y))|^2\,.
$$

First, we argue that $\f_{\tau}$ is $1$-Lipschitz.  This follows because if $\mathcal P(x)=\mathcal P(y)$, then
$$
|d(x, X \setminus \mathcal P(x)) - d(y,X \setminus \mathcal P(y))| = |d(x, X \setminus \mathcal P(x)) - d(y,X \setminus \mathcal P(x))| \leq d(x,y)\,.
$$
On the other hand, if $\mathcal P(x) \neq \mathcal P(y)$, then
$$
d(x, X \setminus \mathcal P(x)), d(y,X \setminus \mathcal P(y)) \leq d(x,y)\,.
$$

Finally, assume that $d(x,y) \geq \tau$ which implies that $d(x,y) > \Delta$.
Since $\mathcal P$ is $\Delta$-bounded, we have $\P[\mathcal P(x) \neq \mathcal P(y)]=1$.
Therefore,
\begin{align*}
\E& \left|\sigma_{\mathcal P(x)} \cdot d(x, X \setminus \mathcal P(x)) - \sigma_{\mathcal P(y)} \cdot d(y, X \setminus \mathcal P(y))\right|^2 \\ &\geq
\P\left[B(x,\e \Delta) \subseteq \mathcal P(x), \sigma_{\mathcal P(x)}=1, \sigma_{\mathcal P(y)}=0\right] \cdot \e^2 \Delta^2 \\
&\geq
\frac{\delta}{4} \e^2 \Delta^2\,.
\end{align*}
It follows that $d(x,y) \geq \tau \implies \|\f_{\tau}(x) - \f_{\tau}(y)\|_{L_2(\Omega',\P)} \geq \frac{\sqrt{\delta} \e \tau}{4}$, completing the proof.
\end{proof}

The next series of results follow from Theorem \ref{thm:partitions} combined with Theorem \ref{thm:main}
and the existence of padded random partitions.  The references for these partitions are
\cite{KPR93} for (i) and (iii), \cite{LS10} for (ii), \cite{GKL03} for (iv), and \cite{NS11} for (v).
We refer to \cite{LN05} for a treatment of (i),(iii), and (iv).

\begin{cor}
There is an absolute constant $K > 0$ such that the following results hold true.
\begin{enumerate}
\item If $(X,d)$ is a planar graph metric, then $M_2(X) \leq K$.
\item If $(X,d)$ is a metric on a graph of (orientable) genus $g > 1$, then $M_2(X) \leq K \log g$.
\item If $(X,d)$ is a metric on a graph which excludes $K_h$ as a minor, then $M_2(X) \leq K h^2$.
\item If $(X,d)$ is $\lambda$-doubling for some $\lambda \geq 1$, then $M_2(X) \leq K \log (1+\lambda)$.
\item If $(X,d)$ has finite Assouad-Nagata dimension, then $X$ has Markov type 2.
\end{enumerate}
\end{cor}

\section{Dimension reduction for martingales in smooth Banach spaces}
\label{sec:unismooth}

The following dimension reduction lemma is a special case
of the continuous-time version proved in \cite{KS91}.
For an exposition of the discrete case,
see \cite[Prop. 5.8.3]{KW92}.

\begin{lemma}\label{lem:KS}
Let $\{M_t\}$ be an $L_2$-valued martingale.
Then there exists an $\mathbb R^2$-valued martingale $\{N_t\}$ such that
for any time $t \geq 0$,
$\|N_t\|_2= \|M_t\|_2$ and $\|N_{t+1}-N_t\|_2 = \|M_{t+1}-M_t\|_2$.
\end{lemma}

We now present a somewhat similar form of martingale dimension reduction for $p$-uniformly smooth Banach spaces.
Consider $p \in (1,2]$ and let $\mathcal Z$ be a $p$-uniformly smooth Banach space with smoothness constant $S_p(\mathcal Z)$.
It is known that for every $z \in \mathcal Z$, there exists a unique functional $J_z \in \mathcal Z^*$
such that the following conditions hold.
\begin{enumerate}
\item $\|J_z\| = \|z\|^{p-1}$
\item $\langle J_z, z\rangle = \|z\|^p$.
\item There is a constant $C = O(S_p(\mathcal Z)^p)$ such that for all $x,y \in \mathcal Z$,
\begin{equation}\label{eq:smooth}
\|x+y\|^p \leq \|x\|^p + p \langle J_x,y\rangle + C \|y\|^p\,.
\end{equation}
\end{enumerate}
See, for instance, \cite[Eq. (3.8)]{XR91} and \cite[Cor 4.17]{Chidume09}.

\begin{lemma}\label{lem:uniformreduce}
For $p \in (1,2]$, the following holds.
Let $\mathcal Z$ be a $p$-uniformly smooth Banach space and let $\{M_t\}$ be a $\mathcal Z$-valued martingale with respect
to the filtration $\{\mathcal F_t\}$.
Then there exists an $\mathbb R^2$-valued martingale $\{N_t\}$ and a constant $K > 0$ such that for any time $t \geq 0$, the following holds.
\begin{enumerate}
\item $\|M_t-M_0\|^p \preceq \|N_t-N_0\|_2^2$, and
\item $\|N_{t+1}-N_{t}\|_2^2 \preceq K \left(\vphantom{\bigoplus}\|M_{t+1}-M_t\|^p + \E\left[\|M_{t+1}-M_t\|^p\mid \mathcal F_{t-1}\right]\right),$
\end{enumerate}
where $K = O(S_p(\mathcal Z)^p)$.
\end{lemma}

\begin{proof}
Suppose $\{M_t\}$ is a martingale with respect to the filtration $\{\mathcal F_t\}$.
We may assume that $\|M_0\|=0$.
Let $\{\e_t\}$ be an independent,
i.i.d. sequence of random signs.  For $z \in \mathbb R^2$, let $z^{\perp} \in \mathbb R^2$ denote a unit vector perpendicular
to $z$.  We define an $\mathbb R^2$-valued process $\{N_t\}$ with $N_0 = (0,0)$.
Let $C$ be the constant from \eqref{eq:smooth}.
For $t \geq 1$,
we put
\begin{align*}
N_t = N_{t-1} & \left(1 + \frac{p}{2} \frac{\langle J_{M_{t-1}}, M_t - M_{t-1} \rangle \1_{A_{t-1}}-\delta_{t-1}}{\|N_{t-1}\|_2^2}\right)
\\
& + \e_t N_{t-1}^{\perp} \left(\sqrt{C+p} \|M_t-M_{t-1}\|^{p/2} + \sqrt{p} \left(\E\left[\|M_t-M_{t-1}\|^p \mid \mathcal F_{t-1}\right]\right)^{1/2}\right)\,.
\end{align*}
where $A_{t-1}$ is the event $\{\|M_t-M_{t-1}\|^p \leq \|N_{t-1}\|_2^2\}$ and $\delta_{t-1} = \E\left[\langle J_{M_{t-1}}, M_t-M_{t-1}\rangle \1_{A_{t-1}}\mid \mathcal F_{t-1}\right]$.  From the definition of $\delta_{t-1}$ and the presence of $\e_t$, it is clear that $\{N_t\}$ is a martingale.

We will prove claim (i) by induction.  It holds trivially for $t=0$.  Now assume that $\|M_{t-1}\|^p \leq \|N_{t-1}\|_2^2$ for some $t > 1$.
Using the definition of $A_{t-1}$,  property (i) of the $J_z$ functional, and our inductive assumption, we have
\begin{eqnarray}
|\langle J_{M_{t-1}}, M_t - M_{t-1}\rangle| \1_{A_{t-1}^c} &\leq & \|M_{t-1}\|^{p-1} \cdot \|M_t - M_{t-1}\| \1_{A_{t-1}^c} \nonumber \\
&\leq & \|N_{t-1}\|^{2(p-1)/p} \cdot \|M_t-M_{t-1}\| \1_{A_{t-1}^c} \nonumber \\
&\leq & \|M_t - M_{t-1}\|^p\,.\label{eq:Jbad}
\end{eqnarray}

Since the martingale property implies $\E[\langle J_{M_{t-1}}, M_t-M_{t-1} \mid \mathcal F_{t-1}\rangle] = 0$, we have
\begin{equation}\label{eq:delta}
|\delta_{t-1}| = \left|\E\left[\langle J_{M_{t-1}}, M_t - M_{t-1}\rangle \1_{A_{t-1}^c} \mid \mathcal F_{t-1}\right]\right| \leq \E\left[ \|M_t - M_{t-1}\|^p \mid \mathcal F_{t-1} \right]\,.
\end{equation}
Now we apply property (iii) of the $J_z$ functional to obtain
\begin{eqnarray}
\label{eq:Jz}
\|M_t\|^p &\leq& \|M_{t-1}\|^p + p \langle J_{M_{t-1}}, M_t - M_{t-1} \rangle + C \|M_t - M_{t-1}\|^p\,.
\end{eqnarray}
On the other hand, using the definition of $N_t$, we have
\begin{align*}
\|N_t\|_2^2 &\geq \|N_{t-1}\|_2^2 + p \left(\langle J_{M_{t-1}}, M_t - M_{t-1} \rangle \1_{A_{t-1}} - \delta_{t-1}\right) \\
&\ \ \ \ \ \ \ \ \ \ \ \ + (C+p) \|M_t - M_{t-1}\|^p + p\, \E\left[\|M_t-M_{t-1}\|^p \mid \mathcal F_{t-1}\right] \\
&\overset{\eqref{eq:delta}}{\geq}
\|N_{t-1}\|_2^2 + p \langle J_{M_{t-1}}, M_t - M_{t-1} \rangle \1_{A_{t-1}} + (C+p)\|M_t-M_{t-1}\|^p \\
&\overset{\eqref{eq:Jbad}}{\geq}
\|N_{t-1}\|_2^2 + p \langle J_{M_{t-1}}, M_t - M_{t-1} \rangle + C \|M_t-M_{t-1}\|^p \\
&\geq
\|M_{t-1}\|^p + p \langle J_{M_{t-1}}, M_t - M_{t-1} \rangle + C \|M_t-M_{t-1}\|^p\,,
\end{align*}
where in the final line we have used the inductive hypothesis.  Combined with \eqref{eq:Jz}, this yields $\|M_t\|^p \leq \|N_t\|_2^2$, verifying claim (i).

We proceed to verify claim (ii).
From the definition of $A_{t-1}$ and the inductive hypothesis, we have
\begin{eqnarray}
|\delta_{t-1}| &=& \left|\E\left[\langle J_{M_{t-1}}, M_t - M_{t-1}\rangle \1_{A_{t-1}} \mid \mathcal F_{t-1}\right]\right| \nonumber \\
&\leq &
\E\left[\|M_{t-1}\|^{p-1} \|M_t - M_{t-1}\| \1_{A_{t-1}} \mid \mathcal F_{t-1}\right]\nonumber\\
&\leq& \|N_{t-1}\|_2^2\,. \label{eq:delta2}
\end{eqnarray}
Additionally, using property (i) of the $J_z$ functional, the inductive hypothesis, and the definition of $A_{t-1}$, we have
\begin{eqnarray*}
\frac{\langle J_{M_{t-1}}, M_t-M_{t-1}\rangle^2 \1_{A_{k-1}}}{\|N_{t-1}\|_2^2} &\leq& \frac{\|M_{t-1}\|^{2(p-1)} \|M_t- M_{t-1}\|^2 \1_{A_{t-1}}}{\|N_{t-1}\|_2^2} \\
&\leq & \frac{\|M_t-M_{t-1}\|^2 \1_{A_{t-1}}}{\|N_{t-1}\|_2^{4(2-p)/p}} \\
&\leq & \|M_t-M_{t-1}\|^p\,.
\end{eqnarray*}
Finally, using \eqref{eq:delta} and \eqref{eq:delta2} yields
$$
\frac{|\delta_{t-1}|^2}{\|N_t\|_2^2} \leq \E\left[ \|M_t - M_{t-1}\|^p \mid \mathcal F_{t-1} \right]\,.
$$
Combining the preceding two inequalities with the definition of $N_t$, we arrive at
$$
\|N_t-N_{t-1}\|_2^2 \leq O(C) \left(\|M_t-M_{t-1}\|^p + \E\left[ \|M_t - M_{t-1}\|^p \mid \mathcal F_{t-1} \right]\right)\,,
$$
completing the verification of claim (ii).
\end{proof}

\section{Non-linear analogs of Kwapien's theorem}
\label{sec:kwapien}

First, we recall a few definitions.
A {\em uniform embedding} between metric spaces is an invertible mapping
such that both it and its inverse are uniformly continuous.
A mapping $f :X \to L_2$ is a {\em coarse embedding} if there are non-decreasing maps
$\alpha,\beta : [0,\infty) \to [0,\infty)$ such that $\beta(t) \to \infty$ and
for all $x,y \in X$,
$$
\beta(d(x,y)) \leq \|f(x)-f(y)\|_2 \leq \alpha(d(x,y))\,.
$$
We begin by proving Theorem \ref{thm:equiv} which we restate
for convenience.

\begin{theorem}\label{thm:iso}
A Banach space $Z$ threshold-embeds into Hilbert space if and only if $Z$ is linearly isomorphic to a Hilbert space.
\end{theorem}

As a tool, we will need the next result which follows from techniques of \cite{Ass83}
and whose proof we defer for a moment.  We recall that if $(X,d)$ is a metric space and
$\e \in (0,1]$, we use $(X,d^{\e})$ to denote the metric space with distance $d^{\e}(x,y) = d(x,y)^{\e}$.

\begin{lemma}[\cite{Ass83}]
\label{lem:ass}
If $(X,d)$ threshold-embeds into $L_2$ then for every $\e \in (0,1)$, the space $(X,d^{1-\e})$ bi-Lipschitz embeds into a Hilbert space.
\end{lemma}

\begin{proof}[Proof of Theorem \ref{thm:iso}]
Since $Z$ threshold-embeds into Hilbert space, it has Markov type 2 by Theorem \ref{thm:main},
hence it also has linear type 2 \cite{Ball92}.
One also has the following alternate and simpler line of argument:  $Z$ has Enflo type 2 by Proposition \ref{prop:enflo}, hence
it also has linear type 2 \cite{Enflo70}.

On the other hand, by Lemma \ref{lem:ass}, $Z$ uniformly embeds
into Hilbert space, and thus by \cite{AMM85} (see also \cite[Cor. 8.17]{BL00}), $Z$ has cotype 2.
Now Kwapien's theorem \cite{Kwapien72} implies that $Z$ is isomorphic to a Hilbert space.
\end{proof}

We now prove Lemma \ref{lem:ass}.  The argument is folklore, but we could not find it
written explicitly.

\begin{proof}[Proof of Lemma \ref{lem:ass}]
Suppose that $\{\varphi_{\tau} : X \to L_2 : \tau > 0 \}$ is a $K$-threshold-embedding of $X$ into $L_2$.
Using \cite[Lem. 5.2]{MN04}, we may assume that $\{\varphi_{\tau} : X \to L_2 : \tau > 0\}$ is a $2K$-threshold-embedding
with the additional property that $\sup_{x \in X} \|\varphi_{\tau}(x)\| \leq \tau$ for all $\tau > 0$.
By scaling, we may assume that each $\f_{\tau}$ is 1-Lipschitz.

Let $\{e_n\}_{n \in \mathbb Z}$ be an orthornormal basis for $\ell_2$ and define
the map $\Phi : X \to L_2 \otimes \ell_2$ by
$$
\Phi(x) = \sum_{n \in \mathbb Z} 2^{-\e n} \f_{2^n} \otimes e_n\,.
$$

Fix $x,y \in X$  and let $m \in \mathbb Z$ be such that $d(x,y) \in [2^{m}, 2^{m+1})$. On the one hand, we have
$$
\|\Phi(x)-\Phi(y)\|_{L_2 \otimes \ell_2} \geq 2^{-\e m} \|\f_{2^m}(x) - \f_{2^m}(y)\| \geq 2^{-\e m} \frac{2^{m}}{2K} \geq \frac{d(x,y)^{1-\e}}{4K}\,.
$$
On the other hand,
\begin{eqnarray*}
\|\Phi(x)-\Phi(y)\|_{L_2 \otimes \ell_2}^2 &=& \sum_{n \in \mathbb Z} 2^{-2\e n} \|\f_{2^n}(x)-\f_{2^n}(y)\|^2 \\
&\leq &
\sum_{n < m} 2^{-2\e n} 2^{2n} + 2^{2(m+1)} \sum_{n \geq m} 2^{-2\e n} \\
&\leq &
O\left(\frac{1}{\e}\right) 2^{2(1-\e)m} \\
&\leq &
O\left(\frac{1}{\e}\right) d(x,y)^{2(1-\e)}\,.
\end{eqnarray*}
\end{proof}

We can now move on to some partial non-linear analogs of Kwapien's theorem.
By \cite{AMM85} and \cite{R06}, respectively,
it is known that if a Banach space $Z$ admits a uniform or coarse
embedding into Hilbert space, then $Z$ has cotype 2.
On the other hand,
there do exist Banach spaces of cotype 2 that do not uniformly or coarse embed
into Hilbert space (e.g. $C_1$ the Schatten trace class);
see \cite[\S 8.2]{BL00}.

\begin{theorem}\label{thm:kwap}
Suppose that a Banach space $Z$ admits a coarse or uniform embedding into a Hilbert space.
Then a subset $X \subseteq Z$ has Markov type 2 if and only if $X$ threshold-embeds
into a Hilbert space.
\end{theorem}

To prove this, we need a few definitions.
Consider a metric space $(X,d)$.   We recall that an {\em $\e$-net} is a maximal subset $N \subseteq X$ such that $d(x,y) \geq \e$ for all $x,y \in N$.
Say that a map $f : X \to L_2$ is {\em $\tau$-thresholding} if for all $x,y \in X$, we have $d(x,y) \geq \tau \implies \|f(x)-f(y)\| \geq \tau$.
Define the parameter
$$
\Lambda_{\tau}(X,\e) \defeq \inf \left\{ \sup_{\stackrel{x,y \in X}{d(x,y) \geq \e \tau}} \frac{\e \|f(x)-f(y)\|_2}{d(x,y)} : \textrm{$f : X \to L_2$ is $\tau$-thresholding} \right\}\,.
$$
Finally, we define $\Lambda(X,\e) = \sup_{\tau > 0} \Lambda_{\tau}(X,\e)$.  Any separable metric space $X$ trivially satisfies $\Lambda(X,\e) \leq 1$.
We say that $(X,d)$ is {\em $\Lambda$-nontrivial} if $\liminf_{\e \to 0} \Lambda(X,\e) = 0$.
Otherwise, we say that $(X,d)$ is {\em $\Lambda$-trivial.}
The next result is straightforward.

\begin{lemma}\label{lem:subset}
If $(X,d)$ is $\Lambda$-nontrivial, then so is a subset $Y \subseteq X$.
\end{lemma}

Now we are in position to use Markov type 2 in conjunction with Ball's extension theorem.

\begin{lemma}\label{lem:lambda}
If $(X,d)$ is $\Lambda$-nontrivial, then $X$ has Markov type 2 if and only if $X$ threshold embeds into $L_2$.
\end{lemma}

\begin{proof}
In light of Theorem \ref{thm:main}, we need only prove the only if direction.
Suppose that $M_2(X) \leq C$.  By Ball's extension theorem (Theorem \ref{thm:ballext}),
there exists a constant $\hat C = O(C)$ such that every Lipschitz
map from a subset of $X$ into $L_2$ admits a Lipschitz extension whose
Lipschitz constant is larger by at most a factor of $\hat C$.

Fix some $\tau > 0$.
Since $X$ is $\Lambda$-nontrivial, there exists an $\e < \frac14$
and a mapping $f : X \to L_2$ such that
for any $(\e\tau)$-net $N \subseteq X$, we have
$\|f|_N\|_{\Lip} \leq \frac{1}{8\hat C \e}$ and which satisfies,
for all $x,y \in X$,
\begin{equation}\label{eq:tau}
d(x,y) \geq \frac{\tau}{2} \implies \|f(x)-f(y)\| \geq \frac{\tau}{2}\,.
\end{equation}
Let $\tilde f : X \to L_2$ be the extension of $f|_N$ guaranteed by Ball's extension theorem,
so that $\|\tilde f\|_{\Lip} \leq  \frac{1}{8\e}$.

Fix any $x,y \in X$ with $d(x,y) \geq \tau$ and let $x',y' \in N$ be such that $d(x,x'), d(y,y') \leq \e \tau$.
In particular, since $\e < \frac14$, we have $d(x',y') \geq \tau/2$.  Using the triangle inequality yields
\begin{eqnarray*}
\|\tilde f(x)-\tilde f(y)\| &\geq & \|\tilde f(x') - \tilde f(y')\| - 2 \e \tau \|\tilde f\|_{\Lip}\\
&\geq &
\|f(x')-f(y')\| - \frac{\tau}{4} \\
&\geq &
\frac{\tau}{2} - \frac{\tau}{4} = \frac{\tau}{4}\,,
\end{eqnarray*}
where in the final line we have used \eqref{eq:tau}.
Since $\tau > 0$ was arbitrary, this proves that $X$ threshold-embeds into $L_2$, completing the proof.
\end{proof}

The next two lemmas, combined with Lemmas \ref{lem:subset} and \ref{lem:lambda}, comprise a proof of Theorem \ref{thm:kwap}.

\begin{lemma}
If a Banach space $Z$ admits a {\em coarse embedding} into $L_2$ then $Z$ is $\Lambda$-nontrivial.
\end{lemma}

\begin{proof}
Let $f : Z \to L_2$ be a coarse embedding with moduli $\alpha,\beta$.
Since $Z$ is a normed space, its metric is convex, and we can assume
that $f$ is Lipschitz for large distances (see, e.g., \cite[Prop. 1.11]{BL00}).
Thus after rescaling $f$ (and the moduli $\alpha,\beta$), we may assume that
\begin{equation}\label{eq:liplarge}
\|x-y\|_Z \geq 1 \implies \|f(x)-f(y)\|_2 \leq \|x-y\|_Z\,.
\end{equation}

Now, by homogeneity, to show that $Z$ is $\Lambda$-nontrivial, it suffices
to show that for every $\delta > 0$, there exist $\tau, \e$ such that
$$
\Lambda_{\tau}(Z,\e) < \delta\,.
$$
To this end, let $\tau > 0$ be chosen large enough so that $\beta(\tau) > \delta^{-1}$,
and define $g = f \cdot \frac{\tau}{\beta(\tau)}$.  Also, we put $\e = 1/\tau$.  First,
if $\|x-y\|_Z \geq \tau$, then $\|g(x)-g(y)\|_2 \geq \frac{\tau}{\beta(\tau)} \|f(x)-f(y)\|_2 \geq \tau$,
hence $g$ is $\tau$-thresholding.

Next, consider any pair $x,y \in Z$ with $\|x-y\|_Z \geq \e \tau = 1$.  Then,
$$
\e \cdot \frac{\|g(x)-g(y)\|_2}{\|x-y\|_Z} = \frac{\|g(x)-g(y)\|_2}{\tau \|x-y\|_Z} = \frac{\|f(x)-f(y)\|_2}{\beta(\tau) \|x-y\|_Z}  \leq \frac{1}{\beta(\tau)} < \delta\,,
$$
where in the penultimate inequality, we have used \eqref{eq:liplarge}.
\end{proof}

\begin{lemma}
If a Banach space $Z$ admits a uniform embedding into $L_2$ then $Z$ is $\Lambda$-nontrivial.
\end{lemma}

\begin{proof}
Let $f : Z \to L_2$ be a uniform embedding.  Since $f^{-1}$ is uniformly continuous,
there exists a $\tau > 0$ such that $\|x-y\|_Z \geq 1 \implies \|f(x)-f(y)\|_2 \geq \tau$.
By rescaling, we may assume that $f$ is $\tau$-thresholding.

Then by a simple metric convexity argument,
we have
$$
\sup_{\stackrel{x,y \in Z}{\|x-y\|_Z \geq \e \tau}} \frac{\e\|f(x)-f(y)\|_2}{\|x-y\|_Z} \leq \sup_{\stackrel{x,y \in Z}{\e\tau/2 \leq \|x-y\|_Z \leq \e \tau}} \frac{2\|f(x)-f(y)\|_2}{\tau}\,.
$$
The latter quantity goes to 0 as $\e \to 0$, since $f$ is uniformly continuous.  Hence $\Lambda_{\tau}(Z,\e) \to 0$.
Thus by homogeneity, $Z$ is $\Lambda$-nontrivial.
\end{proof}

\subsection*{Acknowledgements}

We thank Assaf Naor for preliminary discussions about threshold-embeddings and Kwapien's theorem
in 2005, and valuable comments on early drafts of this manuscript.
We are also grateful to an anonymous referee for relaying
Osekowski's simplified proof of  Lemma \ref{cor-integrate-tail}.

\bibliographystyle{alpha}
\bibliography{mt}

\end{document}